\documentclass[12pt,reqno]{amsart}

\usepackage[T2A]{fontenc}
\usepackage[utf8]{inputenc}
\usepackage[english]{babel}
\usepackage{amsmath,amssymb,amsthm,mathtools}
\usepackage{cite}

\newtheorem{theorem}{THeorem}[section]
\newtheorem{lemma}[theorem]{Lemma}

\newcommand{\End}{\operatorname{End}}
\newcommand{\Sym}{\operatorname{Sym}}
\newcommand{\Span}{\operatorname{span}}

\begin{document}

\title{The Action of the Lie Algebra $\mathfrak{sl}_n$ on Colored Graphs and Multicolored Johnson Graphs}
\author{Leonid Bedratyuk}
\address{ Khmelnytsky National University, Faculty of Information Technology, Ukraine}
\email{leonidbedratyuk@khmnu.edu.ua}

\begin{abstract}
We consider the space of $(n-1)$-colored graphs on a fixed set of
$N$ vertices. Each edge position of the complete graph $K_N$ has $n$
possible states: the absence of an edge and $n-1$ colors. This gives a
natural identification of the space of such graphs with the tensor power
$(\mathbb C^n)^{\otimes m}$, where $m=\binom N2$, and defines on it the
diagonal action of the Lie algebra $\mathfrak{gl}_n$, and, after
restriction, the action of $\mathfrak{sl}_n$.
For a fixed profile
$\alpha=(\alpha_0,\dots,\alpha_{n-1})$, we consider the graph
$J(m;\alpha)$ whose vertices are colored graphs of this profile and whose
adjacency is defined by a single exchange of states in two edge positions.
This graph is the transposition graph on the set of words with fixed
profile, also known as the \emph{multislice}.
The main result is an expression of the adjacency operator in terms of
the root operators of $\mathfrak{sl}_n$ and a derivation of its spectrum
by means of the quadratic Casimir operator of $\mathfrak{gl}_n$ and the
Schur--Weyl decomposition. It is proved that the adjacency operator belongs
to the center of the algebra $\End_{S_m}(\mathcal C_\alpha)$. The
contribution of each spectral block to the multiplicity of the
corresponding eigenvalue is described in terms of a Kostka number and the
dimension of a Specht module.
For $n=2$, one obtains the classical Johnson graph and its known spectrum.
As applications, a formula for the valency is established, connectivity is
proved, the Hoffman bound for independent sets is obtained, and the
three-state case is considered in detail; in this case the natural
symmetrized subspace realizes the module $\Sym^m(\mathbb C^3)$.
\end{abstract}

\maketitle

\section{Introduction}

Algebraic methods play a central role in the spectral theory of graphs,
the theory of association schemes, and algebraic combinatorics
\cite{Delsarte1973,BannaiIto1984,BrouwerCohenNeumaier1989}. They are
especially natural for combinatorial objects that admit a tensor model and
compatible actions of a Lie algebra and a symmetric group.
Let
$$
E=E(K_N),
\qquad
m=|E|=\binom N2.
$$
A simple graph on a fixed set of $N$ vertices is specified by choosing a
subset $S\subseteq E$, that is, each edge position is in one of two
states: the edge is absent or present. Therefore the linear space whose
basis is indexed by all simple graphs is naturally identified with the
tensor power
$
(\mathbb C^2)^{\otimes m}.
$
The operators of adding and deleting one edge, together with a suitably
chosen Cartan operator, form an $\mathfrak{sl}_2$-triple. This point of
view was used in
\cite{BedratyukGraphAlgebraI} to describe the representation-theoretic
structure of the graph algebra.

In this paper we consider a multistate generalization of this construction.
Each edge position may have one of $n$ states
$$
0,1,\dots,n-1,
$$
where the state $0$ means that the edge is absent, and the states
$1,\dots,n-1$ correspond to colors. The local state space is
$\mathbb C^n$, and the space of all $(n-1)$-colored graphs has the tensor
model
$$
\mathcal C_N^{(n)}
\cong
(\mathbb C^n)^{\otimes m}.
$$
The standard diagonal action of $\mathfrak{gl}_n$ on this tensor power has
a direct graph-theoretic interpretation: the matrix unit $E_{ab}$ changes
the state of one edge position from $b$ to $a$. After restriction to the
subalgebra of traceless matrices, we obtain an action of $\mathfrak{sl}_n$.

Let
$$
\alpha=(\alpha_0,\dots,\alpha_{n-1}),
\qquad
\alpha_0+\cdots+\alpha_{n-1}=m,
$$
where $\alpha_a$ is the number of edge positions in state $a$. Denote by
$\mathcal C_\alpha$ the subspace of $\mathcal C_N^{(n)}$ spanned by all
colored graphs of this profile. Equivalently, the standard basis of
$\mathcal C_\alpha$ consists of all words of length $m$ over the alphabet
$\{0,\dots,n-1\}$ in which the symbol $a$ occurs exactly $\alpha_a$ times.

The set of such words is known in the modern literature as the
\emph{multislice}; on it one naturally considers the action of coordinate
transpositions \cite{Chase1973,Ruskey1988,FilmusODonnellWu2022}. In this
paper, the simple graph on the standard basis of the layer
$\mathcal C_\alpha$ is denoted by $J(m;\alpha)$ and is called the
\emph{multicolored Johnson graph}. Its vertices are the words of profile
$\alpha$, and two vertices are adjacent if and only if one is obtained
from the other by a single exchange of two positions containing different
symbols. For $n=2$ and
$
\alpha=(m-k,k)
$
this graph coincides with the classical Johnson graph $J(m,k)$.

The graph $J(m;\alpha)$ differs from the nonbinary Johnson scheme, in
which the total number of nonzero coordinates is fixed, but not the exact
multiplicity of each nonzero symbol
\cite{TarnanenAaltonenGoethals1985}. In our case the entire profile
$\alpha$ is fixed, and the symmetric group $S_m$ acts by permuting
positions. The vertex set of the graph is identified with the space of
left cosets
$$
S_m/Y_\alpha,
\qquad
Y_\alpha
=
S_{\alpha_0}\times\cdots\times S_{\alpha_{n-1}},
$$
where $Y_\alpha$ is the corresponding Young subgroup.

If the profile has two positive coordinates, the corresponding permutation
module is multiplicity-free and the classical Gelfand pair arises. For
profiles with three or more positive coordinates, multiplicity-freeness is
generally lost; the simplest example is the profile $(1,1,1)$. Thus, in
the multistate case, the algebra
$
\End_{S_m}(\mathcal C_\alpha)
$
may be noncommutative.

The main contribution of the paper is the realization of the adjacency
operator $A_\alpha$ of the graph $J(m;\alpha)$ in terms of representations
of Lie algebras. On the weight space $\mathcal C_\alpha$ we prove the
formula
$$
A_\alpha
=
\frac12
\left(
\sum_{a\ne b}\mathcal E_{ab}\mathcal E_{ba}
-
(n-1)mI
\right).
$$
Here $I$ is the identity operator, while $\mathcal E_{ab}$ changes the
state of one edge position from $b$ to $a$ and sums the results of all
possible such changes.
Thus the adjacency operator is expressed through local state-changing
operators generated by the action of $\mathfrak{sl}_n$. After the diagonal
terms are taken into account, this formula relates $A_\alpha$ to the image
of the quadratic Casimir operator of $\mathfrak{gl}_n$.

Using the Schur--Weyl decomposition, we obtain the eigenvalues of the
adjacency operator $A_\alpha$
$$
\theta_{\lambda,\alpha}
=
\frac12
\left(
c_\lambda
-
\sum_{a=0}^{n-1}\alpha_a^2
-
(n-1)m
\right),
$$
where $c_\lambda$ is the eigenvalue of the quadratic Casimir operator on
the irreducible $\mathfrak{gl}_n$-module of highest weight $\lambda$. The
contribution of the block
$$
(V_\lambda^{(n)})_\alpha\otimes S^\lambda
$$
to the multiplicity of the corresponding eigenvalue is
$
K_{\lambda,\alpha}\dim S^\lambda,
$
where $(V_\lambda^{(n)})_\alpha$ is the weight subspace of weight
$\alpha$, $S^\lambda$ is a Specht module, and
$K_{\lambda,\alpha}=\dim (V_\lambda^{(n)})_\alpha$ is the corresponding
Kostka number.
If different partitions $\lambda$ give the same eigenvalue, the full
multiplicity is obtained by summing these contributions.

We also prove that the operator $A_\alpha$ belongs to the center of the
algebra
$
\End_{S_m}(\mathcal C_\alpha),
$
even when this algebra is noncommutative. This result is consistent with
the group-theoretic model of the graph: the adjacency operator differs by
a scalar operator from the image of the sum of all transpositions in the
group algebra $\mathbb C[S_m]$. The role of the action of
$\mathfrak{gl}_n$ is to provide a local operator description of adjacency
and a uniform representation-theoretic derivation of the spectral formula
through the Casimir operator.

The paper is organized as follows. In Section~\ref{preliminaries} we
recall the necessary facts about $\mathfrak{gl}_n$, $\mathfrak{sl}_n$,
tensor powers, weights, and the action of $S_m$. In
Section~\ref{colored_graphs} we construct the space of colored graphs and
the state-changing operators. In Section~\ref{multislice_graph} we define
the graph $J(m;\alpha)$, describe its group-theoretic model, and discuss
its relation to Gelfand pairs. In Section~\ref{sec:spectrum} we obtain
the formula for the adjacency operator, its spectrum, and its centrality.
In Section~\ref{sec:applications} we consider valency, connectivity, the
Hoffman bound, and a numerical example. In Section~\ref{sl3} we study the
three-state case in detail.

\section{Preliminaries}
\label{preliminaries}

In this section we fix notation and recall standard facts from
representation theory; see, for example,
\cite{Humphreys,FultonHarris,FultonYoungTableaux}.

\subsection{The Lie algebras $\mathfrak{gl}_n$ and $\mathfrak{sl}_n$}

Let $\mathfrak{gl}_n=\mathfrak{gl}_n(\mathbb C)$ be the Lie algebra of all
complex $n\times n$ matrices with the commutator
$$
[X,Y]=XY-YX,
$$
and let
$$
\mathfrak{sl}_n
=
\{X\in\mathfrak{gl}_n:\operatorname{tr}X=0\}.
$$
Denote by $E_{ab}$, $0\le a,b\le n-1$, the standard matrix units. They
satisfy the relations
\begin{equation}\label{matrix_commutator}
[E_{ab},E_{cd}]
=
\delta_{bc}E_{ad}-\delta_{ad}E_{cb}.
\end{equation}
As a basis of the Cartan subalgebra in $\mathfrak{sl}_n$ we shall use the
matrices
$$
H_r=E_{r-1,r-1}-E_{rr},
\qquad
r=1,\dots,n-1.
$$

\subsection{Representations of Lie algebras}

Let $\mathfrak g$ be a Lie algebra, and let $V$ be a complex vector
space. A \emph{representation} of the Lie algebra $\mathfrak g$ in $V$ is
a linear map
$$
\rho:\mathfrak g\longrightarrow \End(V),
$$
which preserves the commutator, that is,
$$
\rho([X,Y])
=
[\rho(X),\rho(Y)]
=
\rho(X)\rho(Y)-\rho(Y)\rho(X)
$$
for all $X,Y\in\mathfrak g$. In this case the space $V$ is also called a
$\mathfrak g$-module.

The standard representation of the algebra $\mathfrak{gl}_n$ is its
natural action on the space
$
W=\mathbb C^n,
$
equipped with the standard basis $$ v_0,\dots,v_{n-1}, $$ where $v_i$ is
a coordinate unit vector.
The action is given by ordinary multiplication of a matrix by a vector:
$$
\rho(X)v=Xv,
\qquad
X\in\mathfrak{gl}_n,\quad v\in W.
$$

For the matrix units this action has the form
$$
E_{ab}v_c=\delta_{bc}v_a.
$$
Thus $E_{ab}$ sends the basis state $b$ to the state $a$ and annihilates
all other basis states.

After restriction to the subalgebra $\mathfrak{sl}_n$, we obtain the
standard representation of $\mathfrak{sl}_n$ on the same space $W$.

In what follows, by
$$
\rho_m:\mathfrak{gl}_n\longrightarrow
\End(W^{\otimes m})
$$
we shall denote the representation on the $m$-th tensor power induced by
the standard representation. Its explicit formula will be given in the
next subsection.

\subsection{The diagonal action on a tensor power}

On the space $W^{\otimes m}$ the algebra $\mathfrak{gl}_n$ acts
diagonally:
$$
\rho_m(X)
=
\sum_{s=1}^m X^{(s)},
$$
where $X^{(s)}$ acts as $X$ on the $s$-th tensor factor and as the
identity on the other factors. For the matrix units we put
$$
\mathcal E_{ab}
=
\rho_m(E_{ab})
=
\sum_{s=1}^m E_{ab}^{(s)}.
$$
It follows from \eqref{matrix_commutator} that
\begin{equation}\label{tensor_commutator}
[\mathcal E_{ab},\mathcal E_{cd}]
=
\delta_{bc}\mathcal E_{ad}-\delta_{ad}\mathcal E_{cb}.
\end{equation}

It is convenient to write the standard tensor basis as words
$$
w=c_1c_2\cdots c_m,
\qquad
c_s\in\{0,\dots,n-1\}.
$$
On such a word,
\begin{equation*}
\mathcal E_{ab}(w)
=
\sum_{\substack{1\le s\le m\\c_s=b}}
 c_1\cdots c_{s-1}a c_{s+1}\cdots c_m.
\end{equation*}

If the symbol $a$ occurs in the word $w$ exactly $\mu_a$ times, then
$$
\mathcal E_{aa}w=\mu_aw.
$$
The tuple
$$
\mu=(\mu_0,\dots,\mu_{n-1}),
\qquad
\sum_a\mu_a=m,
$$
is the weight of the word with respect to the diagonal subalgebra of
$\mathfrak{gl}_n$. For $\mathfrak{sl}_n$ the corresponding eigenvalues are
$$
H_rw=(\mu_{r-1}-\mu_r)w,
\qquad
r=1,\dots,n-1.
$$
Since the sum of the coordinates is equal to $m$, these differences
uniquely determine the whole profile $\mu$.

\subsection{The compatible action of the symmetric group}

The symmetric group $S_m$ acts on $W^{\otimes m}$ by permuting the tensor
factors. Denote the corresponding representation by
$$
\pi_m:S_m\longrightarrow\operatorname{GL}(W^{\otimes m}).
$$
For $\sigma\in S_m$ we put
$$
\pi_m(\sigma)
\bigl(
v_{c_1}\otimes\cdots\otimes v_{c_m}
\bigr)
=
v_{c_{\sigma^{-1}(1)}}\otimes\cdots\otimes
v_{c_{\sigma^{-1}(m)}}.
$$
By linear extension we also obtain a representation of the group algebra
$$
\pi_m:\mathbb C[S_m]\longrightarrow\End(W^{\otimes m}).
$$

\begin{lemma}
For all $X\in\mathfrak{gl}_n$ and $\sigma\in S_m$ we have
$$
\pi_m(\sigma)\rho_m(X)
=
\rho_m(X)\pi_m(\sigma).
$$
\end{lemma}

\begin{proof}
For $X^{(s)}$, which acts on the $s$-th factor, we have
$$
\pi_m(\sigma)X^{(s)}\pi_m(\sigma)^{-1}
=
X^{(\sigma(s))}.
$$
Therefore
$$
\sigma\rho_m(X)\sigma^{-1}
=
\sum_{s=1}^m X^{(\sigma(s))}
=
\sum_{t=1}^m X^{(t)}
=
\rho_m(X).
$$
\end{proof}

Thus the image of the universal enveloping algebra
$U(\mathfrak{gl}_n)$ commutes with the action of $S_m$. This compatibility
is the basis of Schur--Weyl duality.

%%%%%%%%%%%%%%%%%%%%%%%%%
\section{The space of colored graphs and the action of $\mathfrak{sl}_n$}
\label{colored_graphs}

In this section we introduce the space of $(n-1)$-colored graphs on a
fixed vertex set and establish its tensor model. Each edge position is
regarded as an independent position with $n$ possible states; therefore
the space of all colored graphs is naturally identified with the tensor
power $(W)^{\otimes m}$. This identification allows us to transfer to
colored graphs the standard diagonal action of the algebra
$\mathfrak{gl}_n$, and, after restriction, the action of
$\mathfrak{sl}_n$. The corresponding matrix operators acquire a direct
combinatorial meaning as operators changing the state of one edge
position.

\subsection{Colored graphs as a tensor space}

Let
$$
V=\{1,\dots,N\},
\qquad
E=E(K_N),
\qquad
m=|E|=\binom N2.
$$
The elements of $E$ will be called \emph{edge positions}. Each such
position may be in one of the states
$$
0,1,\dots,n-1,
$$
where $0$ means the absence of an edge, and $1,\dots,n-1$ is its color.

A basis $(n-1)$-colored graph is specified by a partition
$$
E=S_0\sqcup S_1\sqcup\cdots\sqcup S_{n-1},
$$
where $S_a$ is the set of edge positions in state $a$. The corresponding
basis element will be denoted by
$$
x_{S_0,\dots,S_{n-1}}.
$$
Define
$$
\mathcal C_N^{(n)}
=
\Span_{\mathbb C}
\left\{
 x_{S_0,\dots,S_{n-1}}:
 E=S_0\sqcup\cdots\sqcup S_{n-1}
\right\}.
$$

After enumerating $E=\{e_1,\dots,e_m\}$, a graph corresponds to a word
$c_1\cdots c_m$, where $c_s=a$ if and only if $e_s\in S_a$. Hence we have
a natural identification
\begin{equation}\label{graph_tensor_identification}
\mathcal C_N^{(n)}\cong(W)^{\otimes m}.
\end{equation}
The edge positions here are labeled; we work with graphs on a fixed vertex
set, not with isomorphism classes.

\subsection{State-changing operators}

Via the identification \eqref{graph_tensor_identification}, the operators
$\mathcal E_{ab}$ act on colored graphs. For $a\ne b$ we have
\begin{gather*}
\mathcal E_{ab}
\bigl(x_{S_0,\dots,S_a,\dots,S_b,\dots,S_{n-1}}\bigr)
=
\sum_{e\in S_b}
 x_{S_0,\dots,S_a\cup\{e\},\dots,S_b\setminus\{e\},\dots,S_{n-1}}.
\end{gather*}
Thus $\mathcal E_{ab}$ changes the state of one edge from $b$ to $a$ and
sums all possible such changes. The diagonal operators act by the formula
$$
\mathcal E_{aa}x_{S_0,\dots,S_{n-1}}
=
|S_a|x_{S_0,\dots,S_{n-1}}.
$$
The relations \eqref{tensor_commutator} show that these operators define
an action of $\mathfrak{gl}_n$, while their traceless part defines an
action of $\mathfrak{sl}_n$.

\subsection{The two-state case}

For $n=2$, colored graphs are ordinary simple graphs. If $S\subseteq E$
is the set of present edges, then the corresponding basis element will be
denoted by $x_S$. Put
$$
D_+=\mathcal E_{10},
\qquad
D_-=\mathcal E_{01},
\qquad
H=\mathcal E_{11}-\mathcal E_{00}.
$$
Then
$$
D_+(x_S)=\sum_{e\notin S}x_{S\cup\{e\}},
\qquad
D_-(x_S)=\sum_{e\in S}x_{S\setminus\{e\}},
$$
and
$$
H(x_S)=(2|S|-m)x_S.
$$
A direct verification shows that the relations
$$
[D_+,D_-]=H,
\qquad
[H,D_+]=2D_+,
\qquad
[H,D_-]=-2D_-
$$
hold. Thus these operators realize the Lie algebra $\mathfrak{sl}_2$.

\subsection{Profiles and weight spaces}

Let
$$
\alpha=(\alpha_0,\dots,\alpha_{n-1}),
\qquad
\alpha_a\in\mathbb Z_{\ge0},
\qquad
\sum_a\alpha_a=m.
$$
We call this tuple a \emph{color profile}. Denote by $\mathcal C_\alpha$
the subspace spanned by all basis graphs for which $|S_a|=\alpha_a$:
$$
\mathcal C_\alpha
=
\Span\left\{
 x_{S_0,\dots,S_{n-1}}:
 |S_a|=\alpha_a\text{ for all }a
\right\}.
$$
On this space,
$$
\mathcal E_{aa}=\alpha_aI,
\qquad
H_r=(\alpha_{r-1}-\alpha_r)I.
$$
Hence $\mathcal C_\alpha$ is a weight space and
$$
\mathcal C_N^{(n)}
=
\bigoplus_{\alpha_0+\cdots+\alpha_{n-1}=m}
\mathcal C_\alpha.
$$
For $a\ne b$, the operator $\mathcal E_{ab}$ maps between neighboring
weight spaces:
$$
\mathcal E_{ab}:\mathcal C_\alpha
\longrightarrow
\mathcal C_{\alpha+\varepsilon_a-\varepsilon_b},
$$
and its restriction is zero if $\alpha_b=0$.

\section{The multicolored Johnson graph}
\label{multislice_graph}

In this section, on each weight layer of the space of colored graphs, we
introduce a natural transition graph. Its vertices are colored graphs of
fixed profile, or, equivalently, words with prescribed multiplicities of
symbols, and adjacency is defined by a single exchange of two positions
with different states. We establish the relation of this graph with the
transposition graph on the \emph{multislice}, describe its group-theoretic
model through the action of the symmetric group $S_m$ on the coset space
$S_m/Y_\alpha$, and show that in the two-state case this construction
reduces to the classical Johnson graph.

\subsection{Definition}

The classical Johnson graph $J(m,k)$ has as vertices all $k$-element
subsets of the set $\{1,\dots,m\}$; two vertices are adjacent if one is
obtained from the other by replacing one element. In the language of
characteristic words, this means exchanging one symbol $1$ and one symbol
$0$.

By the \textit{multicolored Johnson graph} $J(m;\alpha)$ we shall mean
the simple graph whose vertices are all words of length $m$ over the
alphabet $\{0,\dots,n-1\}$ with profile $\alpha$. Two vertices are
adjacent if and only if one is obtained from the other by a single
exchange of two positions containing different symbols.

Equivalently, the vertices are the basis colored graphs of the space
$\mathcal C_\alpha$, and adjacency means exchanging the states of two
edge positions. For $n=2$ and $\alpha=(m-k,k)$ we obtain $J(m,k)$.

\subsection{Group-theoretic model}

Recall that the symmetric group $S_m$ acts on the tensor space
$W^{\otimes m}$ by permuting the tensor factors. The corresponding
representation was denoted by
$$
\pi_m:S_m\longrightarrow\operatorname{GL}(W^{\otimes m}).
$$
In the word model this action has the form
$$
\pi_m(\sigma)(c_1c_2\cdots c_m)
=
c_{\sigma^{-1}(1)}
c_{\sigma^{-1}(2)}
\cdots
c_{\sigma^{-1}(m)},
\qquad
\sigma\in S_m.
$$
A permutation of positions does not change the number of occurrences of
each symbol; hence each space $\mathcal C_\alpha$ is invariant under the
action of $S_m$. Denote by
$$
\pi_\alpha
=
\pi_m\big|_{\mathcal C_\alpha}
$$
the corresponding permutation representation on $\mathcal C_\alpha$. In
what follows, the notation $\pi_m$ and $\pi_\alpha$ will also be used for
their linear extensions to the group algebra $\mathbb C[S_m]$.

The action of $S_m$ on the set of words of profile $\alpha$ is
transitive. Indeed, any two words with the same multiplicities of symbols
are obtained from one another by a permutation of positions. Fix the word
$$
w_\alpha
=
\underbrace{0\cdots0}_{\alpha_0}
\underbrace{1\cdots1}_{\alpha_1}
\cdots
\underbrace{(n-1)\cdots(n-1)}_{\alpha_{n-1}}.
$$
Its stabilizer consists of permutations that separately permute positions
occupied by equal symbols. Hence
$$
\operatorname{Stab}_{S_m}(w_\alpha)
=
Y_\alpha
=
S_{\alpha_0}\times\cdots\times S_{\alpha_{n-1}}.
$$
Therefore the map
$$
S_m/Y_\alpha\longrightarrow V(J(m;\alpha)),
\qquad
\sigma Y_\alpha\longmapsto \pi_m(\sigma)w_\alpha,
$$
is a well-defined bijection. After linear extension it gives an
isomorphism of $S_m$-modules
\begin{equation}\label{permutation_module}
\mathcal C_\alpha
\cong
\mathbb C[S_m/Y_\alpha].
\end{equation}

Let $\mathcal T_m$ be the set of all transpositions in $S_m$. The Schreier
graph of the transitive action of $S_m$ on $S_m/Y_\alpha$ with respect to
the set $\mathcal T_m$ has as vertices the left cosets $S_m/Y_\alpha$,
and a transposition $\tau\in\mathcal T_m$ connects the vertex
$\sigma Y_\alpha$ with the vertex $\tau\sigma Y_\alpha$
\cite{Sabatini2022}. Under the bijection \eqref{permutation_module}, the
action of the transposition $\tau=(rs)$ corresponds to exchanging the
symbols in positions $r$ and $s$ of a word of profile $\alpha$.

If equal symbols stand in these positions, the transposition fixes the
word, and therefore a loop arises in the Schreier graph. If the symbols
are different, the resulting word is adjacent to the initial one in the
graph $J(m;\alpha)$. Conversely, each adjacent vertex of $J(m;\alpha)$ is
obtained by the transposition of exactly those two positions in which the
words differ. Hence $J(m;\alpha)$ is the simple undirected graph underlying
the Schreier graph
$$
\operatorname{Sch}
\bigl(S_m\curvearrowright S_m/Y_\alpha,\mathcal T_m\bigr),
$$
after deleting loops and forgetting edge labels. Transposition graphs on
sets of words of fixed profile are considered, in particular, in
\cite{Chase1973,Ruskey1988,FilmusODonnellWu2022}.

If the profile $\alpha$ has at least three positive coordinates, then
$J(m;\alpha)$ is, in general, not a single orbital graph of the action of
$S_m$. The stabilizer $Y_\alpha$ cannot send the exchange of states $a$
and $b$ to the exchange of another pair of states $c$ and $d$, since it
only permutes positions within the classes of equal symbols. Therefore
each unordered pair of states
$$
\{a,b\},
\qquad
\alpha_a\alpha_b>0,
$$
defines a separate orbital graph, and $J(m;\alpha)$ is the union of all
such orbital graphs. In the two-state case there is only one pair of
distinct states, and therefore the classical Johnson graph is a single
orbital graph.

Now consider the element of the group algebra
$$
T_m
=
\sum_{1\le r<s\le m}(rs)
\in\mathbb C[S_m],
$$
that is, the sum of all transpositions. Since the transpositions form one
conjugacy class in $S_m$, we have
$$
T_m\in Z(\mathbb C[S_m]).
$$

The action of this element on the space $\mathcal C_\alpha$ is related to
the adjacency operator by the relation
\begin{equation}\label{class_sum_adjacency}
\pi_\alpha(T_m)
=
A_\alpha
+
\left(
\sum_{a=0}^{n-1}\binom{\alpha_a}{2}
\right)I.
\end{equation}
Indeed, for each state $a$ there are
$$
\binom{\alpha_a}{2}
$$
transpositions of two positions occupied by the symbol $a$. Each such
transposition fixes the initial word and gives a diagonal contribution. By
contrast, a transposition of two positions with different symbols gives
one adjacent vertex of the graph $J(m;\alpha)$, and each adjacent vertex
arises from a unique transposition. Summing over all transpositions gives
\eqref{class_sum_adjacency}.

\subsection{Gelfand pairs and multiplicities}

For $n=2$ and $\alpha=(m-k,k)$ we have
$$
Y_\alpha=S_{m-k}\times S_k
$$
and the classical multiplicity-free decomposition
\begin{equation*}
\mathbb C[S_m/(S_{m-k}\times S_k)]
\cong
\bigoplus_{j=0}^{\min(k,m-k)}S^{(m-j,j)}.
\end{equation*}
Here $S^{(m-j,j)}$ denotes the irreducible Specht module of the symmetric
group $S_m$ corresponding to the partition $(m-j,j)$ of the number $m$.
Thus $(S_m,S_{m-k}\times S_k)$ is a Gelfand pair
\cite{CeccheriniSilbersteinScarabottiTolli}. Multiplicity-freeness
explains why every $S_m$-equivariant operator acts by a scalar on each
irreducible component.

For profiles with three or more nonzero components, this property is not
guaranteed. For example, for $m=3$ and $\alpha=(1,1,1)$ the subgroup
$Y_\alpha$ is trivial, and
$
\mathcal C_\alpha\cong\mathbb C[S_3]
$
is the regular representation. The module $S^{(2,1)}$ has dimension $2$
and occurs in it with multiplicity $2$. Therefore the algebra
$\End_{S_m}(\mathcal C_\alpha)$ may be noncommutative. The operator
formula below shows that the particular operator $A_\alpha$ nevertheless
acts by a scalar on each isotypic block.

%%%%%%%%%%%%%%%%%%%%%%%%%%%%%%%%%%%%%%%%%%%%%
\section{The adjacency operator and its spectrum}
\label{sec:spectrum}

In this section the adjacency operator of the multicolored Johnson graph
is expressed in terms of the operators of the $\mathfrak{gl}_n$-action;
then, by means of the quadratic Casimir operator and the Schur--Weyl
decomposition, its eigenvalues and their multiplicities are determined.

\subsection{Operator formula}
We now write the adjacency operator of the graph $J(m;\alpha)$ in terms
of state-changing operators.
The standard basis elements of the space $\mathcal C_\alpha$ are
identified with the vertices of the graph $J(m;\alpha)$, that is, with
words of profile $\alpha$. The adjacency operator is the linear operator
$$
A_\alpha:\mathcal C_\alpha\longrightarrow\mathcal C_\alpha,
$$
which on each standard basis element $x$ is given by the formula
$$
A_\alpha(x)
=
\sum_{ y\sim x}y.
$$
Here $y\sim x$ means that the vertices $x$ and $y$ are adjacent.
In this basis its matrix is the ordinary adjacency matrix of the graph
$J(m;\alpha)$.

The following formula is central for the subsequent spectral analysis: it
translates combinatorial adjacency into the language of operators generated
by the action of $\mathfrak{gl}_n$ or $\mathfrak{sl}_n$.

\begin{theorem}
On the space $\mathcal C_\alpha$, the adjacency operator $A_\alpha$ of the
graph $J(m;\alpha)$ is given by the formula
\begin{equation}\label{adjacency_formula}
A_\alpha
=
\frac12
\left(
\sum_{a\ne b}\mathcal E_{ab}\mathcal E_{ba}
-(n-1)mI
\right)\Bigg|_{\mathcal C_\alpha}.
\end{equation}
\end{theorem}

\begin{proof}
Let
$$
w=x_{S_0,\dots,S_{n-1}}\in\mathcal C_\alpha.
$$
Consider the action of one summand
$$
\mathcal E_{ab}\mathcal E_{ba},
\qquad a\ne b.
$$
First, the operator $\mathcal E_{ba}$ changes one edge position from
state $a$ to state $b$. There are exactly $\alpha_a$ such positions in the
word $w$. After this, the operator $\mathcal E_{ab}$ changes one edge
position from state $b$ back to state $a$. Two types of contributions
arise here. If the second operator acts on the same position that was
changed by the first operator, then we return to the initial word $w$. For
a fixed ordered pair $(a,b)$ there are exactly $\alpha_a$ such
contributions, that is, they give the diagonal summand
$
\alpha_a w.
$

If the second operator acts on another position which initially was in
state $b$, then the result is a mutual interchange of the states of two
positions: one from $S_a$ and one from $S_b$. These interchanges are
precisely the ones that define the edges of the graph $J(m;\alpha)$.
Hence, for fixed $a\ne b$, we have
$$
\mathcal E_{ab}\mathcal E_{ba}(w)
=
\alpha_a\,w
+
\sum_{e\in S_a,\ f\in S_b}
w^{(e\leftrightarrow f)},
$$
where $w^{(e\leftrightarrow f)}$ denotes the word obtained from $w$ by
interchanging the states in the positions $e$ and $f$.

Now we sum over all ordered pairs $a\ne b$. The diagonal contribution is
$$
\sum_{a\ne b}\alpha_a\,w
=
(n-1)\sum_a\alpha_a\,w
=
(n-1)m\,w.
$$
Each adjacent vertex of the graph $J(m;\alpha)$ is counted twice: once for
the ordered pair $(a,b)$ and once for the ordered pair $(b,a)$. Therefore
$$
\left(
\sum_{a\ne b}\mathcal E_{ab}\mathcal E_{ba}
\right)(w)
=
(n-1)m\,w+2A_\alpha(w).
$$
The formula of the theorem follows immediately.
\end{proof}

For $n=2$ and $\alpha=(m-k,k)$ this formula takes the form
$$
A_{(m-k,k)}
=
\frac12
\left(
\mathcal E_{10}\mathcal E_{01}
+
\mathcal E_{01}\mathcal E_{10}
-mI
\right),
$$
which coincides with the operator description of the Johnson graph through
raising and lowering operators on the Boolean lattice
\cite{Stanley1988,Feinsilver2012}.

On the whole tensor space
$W^{\otimes m}$ the identity
\begin{equation*}
\Omega
=
nmI+2\pi_m(T_m),
\end{equation*}
holds, where
$$
\Omega
=
\sum_{a,b=0}^{n-1}
\mathcal E_{ab}\mathcal E_{ba}
$$
is the image of the quadratic Casimir element in the representation
$\rho_m$ of the algebra $\mathfrak{gl}_n$.

Indeed, the summands in which both operators act on the same tensor factor
give $nmI$, while the summands acting on two different tensor factors give
the operator permuting these factors. Together with
\eqref{class_sum_adjacency}, this gives an independent verification of
the formula \eqref{adjacency_formula}.

\subsection{The quadratic Casimir}

Consider the quadratic element
$$
\Omega_{\mathfrak{gl}_n}
=
\sum_{a,b=0}^{n-1}E_{ab}E_{ba}
\in U(\mathfrak{gl}_n).
$$
With respect to the standard nondegenerate invariant bilinear form
$$
B(X,Y)=\operatorname{tr}(XY)
$$
the bases $\{E_{ab}\}$ and $\{E_{ba}\}$ are mutually dual. Therefore
$\Omega_{\mathfrak{gl}_n}$ is the quadratic Casimir element and, by the
general theory, belongs to the center of the universal enveloping algebra:
$$
\Omega_{\mathfrak{gl}_n}
\in Z\bigl(U(\mathfrak{gl}_n)\bigr),
$$
\cite{Humphreys,FultonHarris}.
In the representation on
$(W)^{\otimes m}$ its image is
$$
\Omega
=
\sum_{a,b=0}^{n-1}\mathcal E_{ab}\mathcal E_{ba}.
$$

\begin{lemma}\label{lem:casimir_eigenvalue}
Let $V_\lambda^{(n)}$ be an irreducible polynomial
$\mathfrak{gl}_n$-module with highest weight
$$
\lambda=(\lambda_0,\dots,\lambda_{n-1}),
\qquad
\lambda_0\ge\cdots\ge\lambda_{n-1}\ge0.
$$
Then $\Omega$ acts on it by the scalar
\begin{equation}\label{eq:casimir_eigenvalue}
c_\lambda
=
\sum_{a=0}^{n-1}
\lambda_a(\lambda_a+n-1-2a).
\end{equation}
\end{lemma}

\begin{proof}
Since $\Omega$ is central, it is enough to compute its action on a highest
weight vector $v_\lambda$. We have
$$
\Omega
=
\sum_aE_{aa}^2
+
\sum_{a<b}(E_{ab}E_{ba}+E_{ba}E_{ab}).
$$
For the standard choice of a Borel subalgebra,
$E_{ab}v_\lambda=0$ for $a<b$, and
$$
E_{ab}E_{ba}v_\lambda
=
[E_{ab},E_{ba}]v_\lambda
=
(E_{aa}-E_{bb})v_\lambda
=
(\lambda_a-\lambda_b)v_\lambda.
$$
Therefore the eigenvalue is
$$
\sum_a\lambda_a^2+
\sum_{a<b}(\lambda_a-\lambda_b)
=
\sum_{a=0}^{n-1}
\lambda_a(\lambda_a+n-1-2a).
$$

Indeed, in the sum
$$
\sum_{a<b}(\lambda_a-\lambda_b)
$$
each $\lambda_a$ occurs with plus sign $n-1-a$ times and with minus sign
$a$ times. Therefore
$$
\sum_{a<b}(\lambda_a-\lambda_b)
=
\sum_{a=0}^{n-1}(n-1-2a)\lambda_a.
$$
Hence
$$
\sum_{a=0}^{n-1}\lambda_a^2+
\sum_{a<b}(\lambda_a-\lambda_b)
=
\sum_{a=0}^{n-1}
\left(
\lambda_a^2+(n-1-2a)\lambda_a
\right)
=
\sum_{a=0}^{n-1}
\lambda_a(\lambda_a+n-1-2a).
$$

\end{proof}

\subsection{Schur--Weyl decomposition and the spectrum}

The classical Schur--Weyl decomposition has the form
\begin{equation}
(W)^{\otimes m}
\cong
\bigoplus_{\substack{\lambda\vdash m\\\ell(\lambda)\le n}}
V_\lambda^{(n)}\otimes S^\lambda,
\end{equation}
where $\ell(\lambda)$ is the number of nonzero parts of the partition
$\lambda$, $V_\lambda^{(n)}$ is an irreducible $\mathfrak{gl}_n$-module,
and $S^\lambda$ is a Specht module of the group $S_m$.

Taking the weight component of profile $\alpha$, we obtain
\begin{equation}\label{weight_schur_weyl}
\mathcal C_\alpha
\cong
\bigoplus_{\substack{\lambda\vdash m\\\ell(\lambda)\le n}}
(V_\lambda^{(n)})_\alpha\otimes S^\lambda.
\end{equation}
The dimension of the weight space is equal to the Kostka number:
$$
\dim(V_\lambda^{(n)})_\alpha=K_{\lambda,\alpha}.
$$

\begin{theorem}\label{spectrum}
On the subspace
$(V_\lambda^{(n)})_\alpha\otimes S^\lambda$, the operator $A_\alpha$ acts
by a scalar with eigenvalue
\begin{equation}\label{eq:spectrum}
\theta_{\lambda,\alpha}
=
\frac12
\left(
 c_\lambda-
 \sum_{a=0}^{n-1}\alpha_a^2-
 (n-1)m
\right),
\end{equation}
where $c_\lambda$ is given by formula
\eqref{eq:casimir_eigenvalue}. Equivalently,
\begin{equation}\label{eq:spectrum_simplified}
\theta_{\lambda,\alpha}
=
\frac12
\left(
\sum_{a=0}^{n-1}\lambda_a(\lambda_a-2a)
-
\sum_{a=0}^{n-1}\alpha_a^2
\right).
\end{equation}
\end{theorem}

\begin{proof}
On $\mathcal C_\alpha$ we have
$$
\sum_{a=0}^{n-1}\mathcal E_{aa}^2
=
\left(\sum_{a=0}^{n-1}\alpha_a^2\right)I.
$$
Therefore
$$
\sum_{a\ne b}\mathcal E_{ab}\mathcal E_{ba}
=
\Omega-
\left(\sum_a\alpha_a^2\right)I.
$$
By Lemma~\ref{lem:casimir_eigenvalue}, the operator $\Omega$ acts on
$V_\lambda^{(n)}$ by the scalar $c_\lambda$. Substitution into
\eqref{adjacency_formula} gives \eqref{eq:spectrum}. Since
$\sum_a\lambda_a=m$, formula \eqref{eq:spectrum_simplified} is
equivalent.
\end{proof}

The contribution of a partition $\lambda$ to the multiplicity of the
eigenvalue $\theta_{\lambda,\alpha}$ is
$$
K_{\lambda,\alpha}\dim S^\lambda.
$$
If different partitions give the same eigenvalue, their contributions are
added. Indeed, the dimension of the block
$(V_\lambda^{(n)})_\alpha\otimes S^\lambda$ is
$K_{\lambda,\alpha}\dim S^\lambda$, and Theorem~\ref{spectrum} states that
$A_\alpha$ acts on this whole block by a single scalar.

\begin{theorem}
The operator $A_\alpha$ belongs to the center of the algebra
$\End_{S_m}(\mathcal C_\alpha)$.
\end{theorem}

\begin{proof}
From \eqref{weight_schur_weyl} it follows that
$$
\End_{S_m}(\mathcal C_\alpha)
\cong
\bigoplus_{\substack{\lambda\vdash m\\K_{\lambda,\alpha}>0}}
\End((V_\lambda^{(n)})_\alpha).
$$
By Theorem~\ref{spectrum}, the operator $A_\alpha$ is scalar on each
multiplicity space $(V_\lambda^{(n)})_\alpha$. Hence in each matrix block
it belongs to the center, and therefore it is central in the whole direct
sum.
\end{proof}

The centrality of $A_\alpha$ can also be derived directly from
\eqref{class_sum_adjacency}. Indeed, since
$$
T_m\in Z(\mathbb C[S_m]),
$$
by Schur's lemma the element $T_m$ acts by a scalar on each irreducible
$S_m$-module $S^\lambda$. In the isotypic decomposition
$$
\mathcal C_\alpha
\cong
\bigoplus_{\lambda\vdash m,\ \ell(\lambda)\le n}
(V_\lambda^{(n)})_\alpha\otimes S^\lambda
$$
the operator $\pi_\alpha(T_m)$ therefore has the form
$$
\pi_\alpha(T_m)\big|_{(V_\lambda^{(n)})_\alpha\otimes S^\lambda}
=
I_{(V_\lambda^{(n)})_\alpha}
\otimes
\tau_\lambda I_{S^\lambda}
$$
for some scalar $\tau_\lambda$.

On the other hand, each operator
$B\in\End_{S_m}(\mathcal C_\alpha)$ on this isotypic block has the form
$$
B\big|_{(V_\lambda^{(n)})_\alpha\otimes S^\lambda}
=
B_\lambda\otimes I_{S^\lambda},
\qquad
B_\lambda\in
\End\bigl((V_\lambda^{(n)})_\alpha\bigr).
$$
Hence $\pi_\alpha(T_m)$ commutes with every $S_m$-equivariant
endomorphism of the space $\mathcal C_\alpha$ and therefore belongs to the
center of the algebra $\End_{S_m}(\mathcal C_\alpha)$. In view of
\eqref{class_sum_adjacency}, the operator $A_\alpha$ differs from
$\pi_\alpha(T_m)$ only by a scalar operator, and hence is also a central
element of $\End_{S_m}(\mathcal C_\alpha)$.

\subsection{Return to the classical Johnson graph}

For $n=2$ the corresponding partitions have the form
$$
\lambda=(m-j,j),
\qquad
0\le j\le\min(k,m-k).
$$
Then
$$
c_\lambda
=
(m-j)(m-j+1)+j(j-1).
$$
For $\alpha=(m-k,k)$, formula \eqref{eq:spectrum} after simplification
gives
$$
\theta_j
=
(k-j)(m-k-j)-j,
$$
which is the known spectrum of the Johnson graph $J(m,k)$; see, for
example, \cite[Theorem 9.1.2]{BrouwerCohenNeumaier1989}

\section{Graph properties and applications}
\label{sec:applications}

In this section, the spectral results are applied to the study of the
basic combinatorial properties of the graph $J(m;\alpha)$. In particular,
we determine the number of its vertices and its valency, prove
connectivity, and derive the Hoffman bound for independent sets.

\subsection{Number of vertices and valency}

The number of vertices of the graph is equal to the multinomial
coefficient
\begin{equation*}
v_\alpha
=
|V(J(m;\alpha))|
=
\dim\mathcal C_\alpha
=
\frac{m!}{\alpha_0!\cdots\alpha_{n-1}!}.
\end{equation*}

\begin{theorem}
The graph $J(m;\alpha)$ is regular of degree
\begin{equation*}
d_\alpha
=
\sum_{0\le a<b\le n-1}\alpha_a\alpha_b
=
\binom m2-
\sum_{a=0}^{n-1}\binom{\alpha_a}{2}.
\end{equation*}
\end{theorem}

\begin{proof}
A neighbor of a word is obtained by choosing two positions with different
symbols and exchanging them. For a pair of states $a<b$, there are
$\alpha_a\alpha_b$ such choices. Summing over all pairs gives the first
formula. The second formula is obtained by removing from all $\binom m2$
pairs those pairs in which the symbols are equal.
\end{proof}

\subsection{Connectivity}

If only one coordinate of the profile $\alpha$ is positive, then it is
equal to $m$, and there is only one word of this profile. Hence in this
case $J(m;\alpha)$ consists of one vertex and is trivially connected. The
nontrivial case is described by the following statement.

\begin{theorem}
If at least two coordinates of the profile $\alpha$ are positive, then the
graph $J(m;\alpha)$ is connected.
\end{theorem}

\begin{proof}
The nontrivial case is when at least two coordinates of the profile are
positive. Let $u$ and $v$ be two words of profile $\alpha$, and let
$$
d_H(u,v)=\#\{s:u_s\ne v_s\}
$$
be their Hamming distance. We argue by induction on $d_H(u,v)$.

If $u\ne v$, choose a position $s$ such that
$u_s=b$, $v_s=a$, and $a\ne b$. Since the number of symbols $a$ in both
words is the same, there exists a position $t\ne s$ such that
$u_t=a$ and $v_t\ne a$. Exchange in $u$ the symbols in positions $s$ and
$t$, and denote the resulting word by $u'$. Then $u'$ is adjacent to $u$,
in position $s$ it already coincides with $v$, and in position $t$ no new
disagreement is created. Hence
$$
d_H(u',v)<d_H(u,v).
$$
By induction, $u'$ is connected to $v$ by a path, and therefore $u$ is
connected to $v$.
\end{proof}

\subsection{The Hoffman bound}
After establishing regularity and connectivity, it is natural to apply
spectral methods to extremal problems on the graph $J(m;\alpha)$. One of
the simplest and most important such problems is estimating the size of an
independent set. In our case, an independent set is a family of colored
graphs of one profile, no two of which are obtained from one another by
interchanging the states of two edge positions. For regular graphs, a
standard tool is the Hoffman bound, which expresses an upper bound in
terms of the smallest eigenvalue of the adjacency matrix.

Assume that the profile $\alpha$ has at least two positive coordinates, so
that $d_\alpha>0$. Denote by
$$
\theta_{\min}
=
\min_{\substack{\lambda\vdash m,\ \ell(\lambda)\le n\\
K_{\lambda,\alpha}>0}}
\theta_{\lambda,\alpha}
$$
the smallest eigenvalue of $A_\alpha$; for a nontrivial graph we have
$\theta_{\min}<0$.

\begin{theorem}\label{thm:hoffman}
For every independent set
$S\subseteq V(J(m;\alpha))$ one has
\begin{equation}\label{eq:hoffman}
|S|
\le
v_\alpha
\frac{-\theta_{\min}}{d_\alpha-\theta_{\min}}.
\end{equation}
\end{theorem}

\begin{proof}
Let $A$ be the adjacency matrix, let $\mathbf1$ be the vector all of whose
coordinates are equal to $1$, and let $\chi_S$ be the indicator vector of
the set $S$. With respect to the standard Euclidean inner product, the
basis vectors indexed by the vertices of the graph are orthonormal. Since
the graph is $d_\alpha$-regular,
$$
A\mathbf1=d_\alpha\mathbf1.
$$
We have
$$
\langle\chi_S,\mathbf1\rangle=|S|,
\qquad
\langle\mathbf1,\mathbf1\rangle=v_\alpha.
$$
Therefore the orthogonal projection of $\chi_S$ onto $\mathbb C\mathbf1$
is equal to $\frac{|S|}{v_\alpha}\mathbf1$, and
$$
\chi_S
=
\frac{|S|}{v_\alpha}\mathbf1+z,
\qquad
z\perp\mathbf1.
$$

Since $S$ is independent,
$$
\langle A\chi_S,\chi_S\rangle=0.
$$
The cross terms in the orthogonal decomposition vanish, hence
$$
0
=
\frac{|S|^2}{v_\alpha}d_\alpha+
\langle Az,z\rangle.
$$
By the definition of $\theta_{\min}$,
$$
\langle Az,z\rangle
\ge
\theta_{\min}\|z\|^2,
$$
and
$$
\|z\|^2
=
|S|-\frac{|S|^2}{v_\alpha}.
$$
Therefore
$$
0
\ge
\theta_{\min}|S|
+
\frac{|S|^2}{v_\alpha}
(d_\alpha-\theta_{\min}).
$$
For $S\ne\varnothing$ we divide by $|S|$ and obtain
\eqref{eq:hoffman}; for the empty set the assertion is obvious.
\end{proof}

For $n=2$ and $\alpha=(m-k,k)$ the eigenvalues satisfy
$$
\theta_{j+1}-\theta_j=2j-m,
$$
and therefore the smallest eigenvalue is attained at
$j=\min(k,m-k)$ and is equal to
$$
\theta_{\min}=-\min(k,m-k).
$$
Thus Theorem~\ref{thm:hoffman} becomes the classical spectral bound for
families of $k$-element subsets no two of which differ by replacing one
element. Equivalently, this is a bound for a constant-weight code with
minimum Hamming distance at least $4$
\cite{GodsilMeagher}.

\subsection{Example $J(3;(1,1,1))$}
We conclude the section with the smallest nontrivial example in which the
multicolored nature of the construction already appears.
Let $N=3$, $n=3$, and $\alpha=(1,1,1)$. Then $m=3$,
$$
v_\alpha
=
|V(J(3;(1,1,1)))|
=
\dim\mathcal C_\alpha
=
\frac{3!}{1!1!1!}=6,
$$
and
$$
d_\alpha=1\cdot1+1\cdot1+1\cdot1=3.
$$
For $n=3$ we have
$$
c_\lambda
=
\lambda_0(\lambda_0+2)+
\lambda_1^2+
\lambda_2(\lambda_2-2).
$$
Since $\sum_a\alpha_a^2=3$ and $(n-1)m=6$,
$$
\theta_{\lambda,\alpha}
=
\frac12(c_\lambda-9).
$$
For the weight $(1,1,1)$, the Kostka number is equal to the number of
standard tableaux of shape $\lambda$, that is,
$K_{\lambda,(1,1,1)}=\dim S^\lambda$. We obtain
$$
\begin{array}{c|c|c|c|c}
\lambda & c_\lambda & \theta_{\lambda,\alpha}
& K_{\lambda,\alpha}\dim S^\lambda & \text{multiplicity}\\ \hline
(3) & 15 & 3 & 1\cdot1 & 1\\
(2,1) & 9 & 0 & 2\cdot2 & 4\\
(1,1,1) & 3 & -3 & 1\cdot1 & 1
\end{array}
$$
Hence
$$
\operatorname{Spec}J(3;(1,1,1))
=
\{3^1,0^4,(-3)^1\}.
$$
The vertices can be identified with the permutations of the symbols
$0,1,2$. Each transposition changes the parity of a permutation, and
therefore the graph is bipartite with two parts of three vertices each.
Since it is $3$-regular, we have
$$
J(3;(1,1,1))\cong K_{3,3}.
$$
The Hoffman bound gives
$$
|S|
\le
6\frac{3}{3+3}=3,
$$
and this bound is sharp.

\section{Three-state edges and the action of $\mathfrak{sl}_3$}
\label{sl3}
In this section we consider the simplest nontrivial case of the general
construction, namely the case of three edge states. It already differs
from the classical two-state model of simple graphs, but remains simple
enough to write all operators explicitly.

Let $n=3$. The states of an edge position have the meaning
$$
0=\text{absent edge},
\qquad
1=\text{first color},
\qquad
2=\text{second color}.
$$
A basis graph is given by a partition
$$
E=S_0\sqcup S_1\sqcup S_2
$$
and is denoted by $x_{S_0,S_1,S_2}$. The six root operators have the
interpretations
$$
\mathcal E_{10}:0\mapsto1,
\qquad
\mathcal E_{01}:1\mapsto0,
$$
$$
\mathcal E_{20}:0\mapsto2,
\qquad
\mathcal E_{02}:2\mapsto0,
$$
$$
\mathcal E_{12}:2\mapsto1,
\qquad
\mathcal E_{21}:1\mapsto2.
$$
For the Cartan elements
$$
H_1=\mathcal E_{00}-\mathcal E_{11},
\qquad
H_2=\mathcal E_{11}-\mathcal E_{22}
$$
we have
$$
H_1x_{S_0,S_1,S_2}
=(|S_0|-|S_1|)x_{S_0,S_1,S_2},
$$
$$
H_2x_{S_0,S_1,S_2}
=(|S_1|-|S_2|)x_{S_0,S_1,S_2}.
$$

\subsection{Symmetrized elements}

For $i,j\ge0$, $i+j\le m$, put
$$
a_{i,j}
=
\sum_{\substack{|S_1|=i,\ |S_2|=j\\S_1\cap S_2=\varnothing}}
 x_{E\setminus(S_1\cup S_2),S_1,S_2}.
$$
If $i<0$, $j<0$, or $i+j>m$, we set $a_{i,j}=0$. Denote
$$
\mathcal S_m^{(3)}
=
\Span\{a_{i,j}:i,j\ge0,\ i+j\le m\}.
$$
This is precisely the subspace of $S_m$-invariant elements in
$(\mathbb C^3)^{\otimes m}$: only the profile $(m-i-j,i,j)$ is retained.

\begin{theorem}\label{thm:sl3_symmetrized}
The subspace $\mathcal S_m^{(3)}$ is $\mathfrak{sl}_3$-invariant, and
\begin{gather*}
\mathcal E_{01}(a_{i,j})=(m-i-j+1)a_{i-1,j},
\qquad
\mathcal E_{10}(a_{i,j})=(i+1)a_{i+1,j},\\
\mathcal E_{02}(a_{i,j})=(m-i-j+1)a_{i,j-1},
\qquad
\mathcal E_{20}(a_{i,j})=(j+1)a_{i,j+1},\\
\mathcal E_{12}(a_{i,j})=(i+1)a_{i+1,j-1},
\qquad
\mathcal E_{21}(a_{i,j})=(j+1)a_{i-1,j+1}.
\end{gather*}
Moreover,
$$
H_1(a_{i,j})=(m-2i-j)a_{i,j},
\qquad
H_2(a_{i,j})=(i-j)a_{i,j}.
$$
\end{theorem}

\begin{proof}
Consider, for example, $\mathcal E_{01}$. Its target profile is
$(m-i-j+1,i-1,j)$. For a fixed basis element of this profile, a preimage
is obtained by choosing one of the $m-i-j+1$ zero positions and changing
its state to $1$. Therefore each target element arises with multiplicity
$m-i-j+1$, whence
$$
\mathcal E_{01}(a_{i,j})=(m-i-j+1)a_{i-1,j}.
$$
For $\mathcal E_{10}$, each target element has $i+1$ positions in state
$1$, each of which could have been zero in the preimage; hence
$$
\mathcal E_{10}(a_{i,j})=(i+1)a_{i+1,j}.
$$
The same counts for the transitions $2\leftrightarrow0$ and
$1\leftrightarrow2$ give the remaining four formulas. The diagonal
formulas follow from the profile $(m-i-j,i,j)$. All images belong to
$\mathcal S_m^{(3)}$, and hence this subspace is invariant.
\end{proof}

\subsection{Relation with the symmetric power}
The obtained formulas have a standard interpretation in terms of the
polynomial model of representations. The space generated by the elements
$a_{i,j}$ can be identified with the space of homogeneous polynomials of
degree $m$ in three variables. Under this identification, the numbers of
edges in the states $0,1,2$ correspond to the degrees of the variables
$u_0,u_1,u_2$.

\begin{theorem}
As an $\mathfrak{sl}_3$-module,
$$
\mathcal S_m^{(3)}\cong\Sym^m(\mathbb C^3).
$$
\end{theorem}

\begin{proof}
Define
$$
\Phi(a_{i,j})
=
\frac{u_0^{m-i-j}}{(m-i-j)!}
\frac{u_1^i}{i!}
\frac{u_2^j}{j!}.
$$
The images of the elements $a_{i,j}$ form a divided-power basis in the
space of homogeneous polynomials of degree $m$. In the polynomial model,
$E_{ab}$ acts as $u_a\frac{\partial}{\partial u_b}$.

For example,
$$
u_0\frac{\partial}{\partial u_1}
\left(
\frac{u_0^{m-i-j}}{(m-i-j)!}
\frac{u_1^i}{i!}
\frac{u_2^j}{j!}
\right)
=
\frac{u_0^{m-i-j+1}}{(m-i-j)!}
\frac{u_1^{i-1}}{(i-1)!}
\frac{u_2^j}{j!}.
$$
Since
$$
\frac{u_0^{m-i-j+1}}{(m-i-j)!}
=
(m-i-j+1)
\frac{u_0^{m-i-j+1}}{(m-i-j+1)!},
$$
we have
$$
u_0\frac{\partial}{\partial u_1}\Phi(a_{i,j})
=
(m-i-j+1)\Phi(a_{i-1,j}).
$$
This corresponds to the formula
$$
\mathcal E_{01}(a_{i,j})
=
(m-i-j+1)a_{i-1,j}.
$$
All other formulas of Theorem~\ref{thm:sl3_symmetrized} are obtained
similarly.

Thus $\Phi$ is not only an isomorphism of vector spaces, but also an
isomorphism of $\mathfrak{sl}_3$-modules.
\end{proof}

Thus the constructed symmetrized subspace $\mathcal S_m^{(3)}$ realizes
the irreducible $\mathfrak{sl}_3$-module
$
\Sym^m(\mathbb C^3).
$
This is one of the simplest families of irreducible
$\mathfrak{sl}_3$-modules, whose weight diagrams have triangular form.
At the same time, general irreducible $\mathfrak{sl}_3$-modules have
weight diagrams of hexagonal form, \cite{FultonHarris}. Their realization
within the graph model requires a broader construction than the
symmetrized subspace considered here.

\section{Conclusions}

In this paper, we constructed natural actions of $\mathfrak{gl}_n$ and
$\mathfrak{sl}_n$ on the space of $(n-1)$-colored graphs on a fixed
vertex set. Graphs of a fixed profile form the weight space
$\mathcal C_\alpha$, while transpositions of the states of two edge
positions define the graph $J(m;\alpha)$ on its standard basis. This graph
is the known transposition graph on the multislice; the proposed model
relates it to local edge-state-changing operators.

The main operator result has the form
$$
A_\alpha
=
\frac12
\left(
\sum_{a\ne b}\mathcal E_{ab}\mathcal E_{ba}
-(n-1)mI
\right)\Bigg|_{\mathcal C_\alpha}.
$$
Using the quadratic Casimir of $\mathfrak{gl}_n$ and the Schur--Weyl
decomposition, this formula yields the complete spectrum, its
multiplicities, and the centrality of $A_\alpha$ in
$\End_{S_m}(\mathcal C_\alpha)$. For $n=2$, the construction reduces to
the classical Johnson graph and the standard $\mathfrak{sl}_2$-action on
the Boolean lattice.

As graph-theoretic consequences, formulas for the number of vertices and
the valency were obtained, a self-contained combinatorial proof of
connectivity was given, and the Hoffman bound was applied to independent
sets. The three-state case shows that the symmetrized subspace naturally
realizes the irreducible $\mathfrak{sl}_3$-module
$\Sym^m(\mathbb C^3)$. Further research may be directed toward the
interaction of the constructed action with relabeling the vertices of the
underlying graph and toward realizing broader families of irreducible
$\mathfrak{sl}_n$-modules in spaces of colored graphs.  
Further research may also address Markov chains generated by random
state transpositions on the multislice, including the determination of
their spectral gaps and mixing properties using the operator framework
developed here.


\begin{thebibliography}{99}

\bibitem{BannaiIto1984}
E.~Bannai and T.~Ito,
\textit{Algebraic Combinatorics I: Association Schemes},
Benjamin/Cummings, Menlo Park, 1984.

\bibitem{BedratyukGraphAlgebraI}
L.~Bedratyuk,
\textit{The Graph Algebra I: Representation-Theoretic Structure},
arXiv:2606.29558, 2026.

\bibitem{BrouwerCohenNeumaier1989}
A.~E. Brouwer, A.~M. Cohen and A.~Neumaier,
\textit{Distance-Regular Graphs},
Springer-Verlag, Berlin, 1989.

\bibitem{CeccheriniSilbersteinScarabottiTolli}
T.~Ceccherini-Silberstein, F.~Scarabotti and F.~Tolli,
\textit{Harmonic Analysis on Finite Groups: Representation Theory,
Gelfand Pairs and Markov Chains},
Cambridge University Press, Cambridge, 2008.

\bibitem{Chase1973}
P.~J. Chase,
\textit{Transposition graphs},
SIAM Journal on Computing, \textbf{2} (1973), no.~2, 128--133.

\bibitem{Delsarte1973}
P.~Delsarte,
\textit{An algebraic approach to the association schemes of coding theory},
Philips Research Reports Supplements, no.~10, 1973, 1--97.

\bibitem{Feinsilver2012}
P.~Feinsilver,
\textit{Representations of $\mathfrak{sl}_2$ in the Boolean lattice, and
the Hamming and Johnson schemes},
Infinite Dimensional Analysis, Quantum Probability and Related Topics,
\textbf{15} (2012), no.~3, 1250019.

\bibitem{FilmusODonnellWu2022}
Y.~Filmus, R.~O'Donnell and X.~Wu,
\textit{Log-Sobolev inequality for the multislice, with applications},
Electronic Journal of Probability, \textbf{27} (2022), Paper No.~33,
1--30.

\bibitem{FultonHarris}
W.~Fulton and J.~Harris,
\textit{Representation Theory: A First Course},
Springer-Verlag, New York, 1991.

\bibitem{FultonYoungTableaux}
W.~Fulton,
\textit{Young Tableaux: With Applications to Representation Theory and
Geometry},
Cambridge University Press, Cambridge, 1997.

\bibitem{GodsilMeagher}
C.~Godsil and K.~Meagher,
\textit{Erd\H{o}s--Ko--Rado Theorems: Algebraic Approaches},
Cambridge University Press, Cambridge, 2016.



\bibitem{Sabatini2022}
L.~Sabatini,
\textit{Random Schreier graphs and expanders},
Journal of Algebraic Combinatorics, \textbf{56} (2022), 889--901.



\bibitem{Humphreys}
J.~E. Humphreys,
\textit{Introduction to Lie Algebras and Representation Theory},
Springer-Verlag, New York, 1972.

\bibitem{Ruskey1988}
F.~Ruskey,
\textit{A Hamilton path in the transposition graph for multiset permutations},
Congressus Numerantium, \textbf{67} (1988), 27--34.

\bibitem{Stanley1988}
R.~P. Stanley,
\textit{Differential posets},
Journal of the American Mathematical Society, \textbf{1} (1988), no.~4,
919--961.

\bibitem{TarnanenAaltonenGoethals1985}
H.~Tarnanen, M.~J. Aaltonen and J.-M.~Goethals,
\textit{On the nonbinary Johnson scheme},
European Journal of Combinatorics, \textbf{6} (1985), no.~3, 279--285.

\end{thebibliography}
\end{document}